\documentclass[reqno]{amsart}
\usepackage{fullpage}
\usepackage{array}
\usepackage{mathrsfs}
\usepackage{multirow}
\usepackage{hyperref}
\usepackage{graphicx,color}
\usepackage{float}
\usepackage{ytableau}
\usepackage{amsmath}
\usepackage{comment}
\usepackage{amssymb}

\newtheorem{theorem}{Theorem}[section]

\newtheorem{lemma}[theorem]{Lemma}
\newtheorem{corollary}[theorem]{Corollary}
\newtheorem{remark}[theorem]{Remark}
\newtheorem{proposition}[theorem]{Proposition}
\theoremstyle{definition}
\newtheorem{definition}[theorem]{Definition}

\newcommand{\Id}{\operatorname{Id}}
\newcommand{\var}{\operatorname{var}}

\newcommand{\simeqT}{\mathrel{\sim_T}}

\newcommand{\eps}{\varepsilon}

\usepackage{xcolor}

\newcommand{\black}{\color{black}}

\begin{document}
\title{Varieties of associative algebras  with quadratic codimension growth}

\author{Wesley Q. Cota$^{a,1,*}$}
\author{Luiz H. S. Matos$^{b,2}$}
\author{Pedro Quintino$^{c}$}

\thanks{{\it E-mail address:} wesleyqc@ime.usp.br (Cota), luizmatos@ufmg.br (Matos), pedro.quintino@unimontes.br (Quintino).}

\dedicatory{$^a$ IME, USP, Rua do Matão 1010, 05508-090, São Paulo, Brazil 
\\  $^b$ ICEx, UFMG, Avenida Antonio Carlos 6627, 31123-970, Belo Horizonte, Brazil \\
$^c$ CCET, Unimontes, Avenida Rui Braga, 39401-089, Montes Claros,
Brazil}

\thanks{\footnotesize $^1$ Supported by FAPESP, grant no.~2025/05699-0 and 2026/06588-0.}
\thanks{\footnotesize $^{2}$ Partially supported by FAPEMIG}
\thanks{\footnotesize $^{*}$ Corresponding author}



\begin{abstract}
We classify, up to PI-equivalence, associative algebras over a field of characteristic zero whose codimension sequences have quadratic growth. More precisely, making use of fundamental algebras, we prove that every such algebra is PI-equivalent to a finite direct sum of algebras generating minimal varieties of at most quadratic codimension growth, together with a possible nilpotent summand.

\end{abstract}

\maketitle

\section{Introduction}
Determining the polynomial identities of an associative algebra is difficult even for matrix algebras. Over a field $F$ of characteristic zero, every polynomial identity of an algebra $A$ is a consequence of its multilinear identities. This leads to the study of the codimension sequence
\[
   c_n(A)=\dim_F\frac{P_n}{P_n\cap\operatorname{Id}(A)},
   \qquad n\geq1,
\]
where $P_n$ is the space of multilinear polynomials in $n$ variables and $\operatorname{Id}(A)$ is the $T$-ideal of polynomial identities of $A$. The number $c_n(A)$ counts the linearly independent multilinear polynomials modulo the identities of $A$. Its asymptotic behavior gives a quantitative measure of the size of the variety generated by $A$.

Regev \cite{RG} proved that, for any PI-algebra $A$, the general bound $c_n(A)\leq \dim_F P_n=n!$ can be replaced by an exponential bound. Later, Kemer \cite{Kem1} showed that these codimensions are either polynomially bounded or grow exponentially. Subsequently, Giambruno and Zaicev \cite{GiZai} proved that the exponential rate of growth of $c_n(A)$ exists and is a nonnegative integer, called
the PI-exponent of $A$. For nonnilpotent algebras with polynomial codimension growth, this exponent is always one and therefore does not distinguish the different growth rates within this class. The degree of polynomial growth thus provides a finer invariant, leading naturally to the problem of classifying varieties
according to this degree.

A variety of polynomial codimension growth is called \emph{minimal} if every proper subvariety has codimension growth of strictly smaller degree. In particular, a variety of quadratic growth is minimal if every proper subvariety has at most linear growth. The study of minimal varieties is motivated by the question of whether they suffice to describe arbitrary varieties of polynomial growth. More precisely, is every algebra of polynomial codimension growth PI-equivalent to a finite direct sum of algebras generating minimal varieties of polynomial growth, together with a possible nilpotent summand? This question remains open in general. Here, we recall that two algebras are PI-equivalent if they satisfy the same set of polynomial identities.

For varieties generated by unitary algebras, several classification results are known. Giambruno, La Mattina and Petrogradsky \cite{Petro} classified, up to PI-equivalence, the unitary algebras whose codimensions have at most cubic growth. In \cite{Mara3}, de Oliveira and Vieira obtained the corresponding classification for growth of degree four. In the same setting, Giambruno, La Mattina and Zaicev \cite{GiamLaMaZai} characterized minimal varieties of arbitrary polynomial growth and gave a procedure for constructing their $T$-ideals. These results concern varieties generated by unitary algebras and leave open the corresponding classification problems without this assumption.

For arbitrary associative algebras, Giambruno and La Mattina \cite{GLM2005} classified the varieties of at most linear codimension growth and identified the two minimal varieties of linear growth, generated by $A_1$ and $A_1^*$. Subsequently, Alves Jorge and Vieira \cite{JorgeVieira2006}
proved that
\[
   \operatorname{var}(A_2),\qquad
   \operatorname{var}(A_4),\qquad
   \operatorname{var}(A_4^*),\qquad
   \operatorname{var}(A_5),\qquad
   \operatorname{var}(A_6)
\]
are precisely the five minimal varieties of quadratic growth. These results left open the question of whether every variety of quadratic growth can be generated by a finite direct sum of some of these seven algebras, possibly together with
a nilpotent summand.

In this paper, we give an affirmative answer to this question. Let
\[
   \mathcal L=\{A_1,A_1^*\}
   \quad\text{and}\quad
   \mathcal Q=\{A_2,A_4,A_4^*,A_5,A_6\},
\]
where the algebras in these families are recalled in Section~3. Our main result is the following.

\begin{theorem}
\label{thm:introduction-main}
Let $\mathcal V$ be a variety of associative algebras over a field of characteristic zero. Then $\mathcal V$ has quadratic codimension growth if and only if it is generated by an algebra of the form
\[
   N\oplus\bigoplus_{M\in\mathcal S}M,
\]
where $N$ is a nilpotent algebra and
$\mathcal S\subseteq\mathcal L\cup\mathcal Q$ satisfies
$\mathcal S\cap\mathcal Q\neq\emptyset$.
\end{theorem}

This theorem extends the classification of minimal varieties obtained in 2006 to all varieties of quadratic growth. At the same time, it establishes that the known minimal generators are sufficient in degree two, answering the general question above for quadratic codimension growth.

The proof uses the theory of fundamental algebras. Every finite-dimensional algebra is PI-equivalent to a finite direct sum of fundamental algebras \cite{AGP2020}. The cocharacter characterization of these algebras obtained by Giambruno, Polcino Milies and Zaicev \cite{GPMZ2020}, together with the decomposition and fusion results of Giambruno, Quintino and Vieira \cite{GQV2026}, provides a reduction of the classification problem to a single fundamental algebra, apart from a nilpotent summand. In the quadratic case, the relevant fundamental algebra has radical of nilpotency index three.

The main step is to show how the radical components and their products determine the polynomial identities of this fundamental algebra. We construct a direct sum $B(A)$ of algebras from $\mathcal L\cup\mathcal Q$ and prove that
\[
   \operatorname{Id}(A)=\operatorname{Id}(B(A)).
\]
The quadratic growth ensures that at least one summand belongs to $\mathcal Q$. Combining this description with the reduction to fundamental algebras yields Theorem~\ref{thm:introduction-main}.

The paper is organized as follows. Section~2 recalls the necessary facts on codimensions, cocharacters and fundamental algebras. Section~3 reviews the model algebras and the minimal varieties of quadratic growth. Section~4 constructs the algebra $B(A)$ and compares its identities with those of $A$. The final section provides the complete proof of Theorem \ref{thm:introduction-main}.

\black

\section{PI-algebras and fundamental algebras}
\label{sec:preliminaries}

Throughout the paper, \(F\) denotes a field of characteristic zero and all algebras are associative \(F\)-algebras, not necessarily unital. Let
\[
F\langle X\rangle=F\langle x_1,x_2,\ldots\rangle
\]
be the free associative algebra generated by a countable set \(X=\{x_1,x_2,\ldots\}\) of variables. A polynomial
\(f=f(x_1,\ldots,x_n)\in F\langle X\rangle\) is called a polynomial
identity of an \(F\)-algebra \(A\) if
$f(a_1,\ldots,a_n)=0,$ for all $a_1,\ldots,a_n\in A.$ In this case, we write \(f\equiv 0\) on \(A\). The set of all polynomial
identities of \(A\) is denoted by
\[
\operatorname{Id}(A)
=
\{f\in F\langle X\rangle\mid f\equiv 0\text{ on }A\}.
\] This set is an ideal of $F\langle X\rangle$ invariant under every endomorphism of the free algebra and is therefore a $T$-ideal, which is called the $T$-ideal of $A$. Since $\operatorname{char}F=0$,
every $T$-ideal is determined by its multilinear polynomials.

A class \(\mathcal V\) of associative \(F\)-algebras is called a
\emph{variety} if there exists a set of polynomials
\(S\subseteq F\langle X\rangle\) such that $\mathcal V
=
\{B\mid f\equiv 0\text{ on }B
\text{ for every }f\in S\}.$ For an \(F\)-algebra \(A\), we denote by
\[
\operatorname{var}(A)
=
\{B\mid \operatorname{Id}(A)\subseteq\operatorname{Id}(B)\}
\]
the variety generated by \(A\). Two algebras \(A\) and \(B\) are said
to be \emph{PI-equivalent} (or \emph{$T$-equivalent}) if
$\operatorname{Id}(A)=\operatorname{Id}(B)$, and we write \(A\sim_T B\) in this case.

\subsection{Codimensions, cocharacters and polynomial growth}

For $n\geq 1$, let $P_n$ be the space of multilinear polynomials in
$n$ variables and set
\[
    P_n(A)=\frac{P_n}{P_n\cap\operatorname{Id}(A)}.
\]
The integer  $c_n(A)=\dim_F P_n(A)$ is the $n$th codimension of $A$.  We say that $A$ has \emph{polynomial
codimension growth} if $c_n(A)\leq an^b$ for  suitable constants $a,b>0$ and
all $n\geq1$.  More specifically, we say that $A$ has codimension growth of
degree $d$ if there exist constants $a,b>0$ and $n_0$ such that
\[
       an^d\leq c_n(A)\leq bn^d
       \qquad\text{for every }n\geq n_0.
\]

For $\mathcal{V}=\textnormal{var}(A)$ we define $c_n(\mathcal{V})=c_n(A)$, for all $n\geq 1$.

Regev \cite{RG} proved that the sequence of codimensions of every PI-algebra is exponentially
bounded.  Kemer subsequently established that, in characteristic zero, the
polynomial case is separated from the exponential case by two distinguished
algebras.  Let $\mathcal G$ be the infinite-dimensional Grassmann algebra and
let $UT_2$ be the algebra of $2\times2$ upper triangular matrices over $F$.

\begin{theorem}[\cite{Kem1}]
\label{thm:kemer-polynomial-growth}
Let $A$ be a PI-algebra over a field of characteristic zero.  Then the
codimension sequence of $A$ is polynomially bounded if and only if $\mathcal G,UT_2\notin\operatorname{var}(A).$
\end{theorem}

Consequently, over a field of characteristic zero, the codimension sequence of every PI-algebra is either polynomially bounded or grows exponentially.

Another consequence of Kemer's theory is the following representability result (see~\cite{GiambrunoZai}). For affine representability, see also~\cite{AKK2016}.

\begin{theorem}
\label{thm:representability}
Let $\mathcal V$ be a variety of associative algebras over a field of
characteristic zero.  Then either $\mathcal G\in\mathcal V$ or there exists a finite-dimensional algebra $A$ such that $\mathcal V=\operatorname{var}(A)$.
\end{theorem}

As a consequence, we obtain the following.

\begin{corollary}
\label{cor:finite-dimensional-reduction}
Every variety of polynomial codimension growth over a field of
characteristic zero is generated by a finite-dimensional algebra.
\end{corollary}

The action of the symmetric group $S_n$ on $P_n$ induces an $S_n$-module
structure on $P_n(A)$.  Its character, denoted by $\chi_n(A)$, is called the $n$th
cocharacter of $A$.  By complete reducibility,
\begin{equation}
\label{eq:cocharacter-decomposition}
       \chi_n(A)=\sum_{\lambda\vdash n}m_\lambda(A)\chi_\lambda,
\end{equation}
where $\chi_\lambda$ is the irreducible $S_n$-character corresponding to a partition $\lambda$ of $n$ and $m_\lambda(A)$ is its multiplicity.  Consequently,
\begin{equation*}
       c_n(A)=\sum_{\lambda\vdash n}m_\lambda(A)d_\lambda,
       \qquad d_\lambda=\chi_\lambda(1).
\end{equation*}
The $n$th colength of $A$ is
\[
       \ell_n(A)=\sum_{\lambda\vdash n}m_\lambda(A).
\]
Thus the colength counts the irreducible constituents of the cocharacter.

For every finite-dimensional algebra \(A\), we denote by \(J=J(A)\) its Jacobson radical and set
\[
s_A:=\min\{s\geq 0\mid J(A)^{s+1}=\{0\}\}.
\]
Thus \(J(A)^{s_A+1}=\{0\}\), and, whenever \(J(A)\neq \{0\}\) we have $J(A)^{s_A}\neq \{0\}.$

The following theorem gives a representation-theoretic characterization of polynomial codimension growth. In particular, it provides the uniform bound on cocharacter multiplicities used below.

\begin{theorem}[\cite{GZ2000,MRZ1999}]
\label{thm:bounded-colength}
Let $A$ be a PI-algebra over a field of characteristic zero. The following
conditions are equivalent:
\begin{enumerate}
    \item $A$ has polynomial codimension growth;
    
    \item there exist a finite-dimensional algebra \(B\) and an integer \(q\geq 1\) such that
$\operatorname{var}(A)=\operatorname{var}(B)$, $J(B)^q=\{0\}$ and
\[
\chi_n(A)=\chi_n(B)
=\sum_{\substack{\lambda\vdash n\\ n-\lambda_1<q}}
m_\lambda(B)\chi_\lambda,
\quad n\geq 1.
\]
    
    \item There exists a
    constant $M>0$ such that   $\ell_n(A)\leq M$ for every $n\geq1.$
\end{enumerate}
In particular, whenever these conditions hold there exists $M>0$ such that $ m_\lambda(A)\leq M$
for every $n\geq1$ and every partition $\lambda\vdash n$.
\end{theorem}

If $A$ is finite-dimensional, then in item~(2) we may take $B=A$, see \cite[Theorem 3]{GZ2000}. Since, by definition,
$J(A)^{s_A+1}=\{0\},$ we may choose $q=s_A+1$. Consequently, for every $n\geq1$ and
every partition $\lambda\vdash n$,
$m_\lambda(A)\neq0$ implies  $n-\lambda_1\leq s_A.$ This observation will be used repeatedly below.

Let $K$ be a field extension of $F$ and set
\[
A_K:=K\otimes_F A.
\]
Since $\operatorname{char}F=0$, every $T$-ideal is determined by
its multilinear elements. Under the natural identification
$K\langle X\rangle\cong K\otimes_F F\langle X\rangle,$ one has
\[
\operatorname{Id}_K(A_K)
=
K\otimes_F\operatorname{Id}_F(A).
\]

Consequently, for every $n\geq1$ there is an isomorphism of
$KS_n$-modules $P_n(A_K)\cong K\otimes_F P_n(A).$ It follows that
$c_n(A_K)=c_n(A)$ and $m_\lambda(A_K)=m_\lambda(A)$ for every $\lambda\vdash n$. Thus the codimension and
cocharacter sequences are unchanged by scalar extension.

Moreover, if $A$ is finite-dimensional, then, since every field
of characteristic zero is perfect,
$J(A_K)=K\otimes_FJ(A).$ Hence the nilpotency index of the radical is preserved, and so
\[
s_{A_K}=s_A
\qquad\text{and}\qquad
\dim_K\frac{A_K}{J(A_K)}
=
\dim_F\frac{A}{J(A)}.
\]

Taking $K=\overline F$, we may therefore assume from now on, without loss of generality, that the ground field is
algebraically closed.

Corollary~\ref{cor:finite-dimensional-reduction} reduces the study of varieties of polynomial codimension growth to finite-dimensional algebras. The next step is to single out those finite-dimensional algebras that cannot, at the level of polynomial identities, be replaced by algebras of smaller structural complexity.

\subsection{Fundamental algebras}

Let $A$ be a finite-dimensional algebra.  We fix a Wedderburn--Malcev
decomposition
\begin{equation}
\label{eq:wedderburn-malcev}
       A=\overline A+J(A),
       \qquad
       \overline A=A_1\oplus\cdots\oplus A_q,
\end{equation}
where $J(A)$ is the Jacobson radical and the $A_i$ are simple algebras. We set \(t_A:=\dim_F\overline A\). Adopting the convention \(J(A)^0=A\), we recall that the integer \(s_A\) is characterized by
\[
J(A)^{s_A}\neq \{0\}
\qquad\text{and}\qquad
J(A)^{s_A+1}=\{0\}.
\]

We now recall the definition of a fundamental algebra in the form used in
\cite{GQV2026}.  Suppose first that $A$ is not simple and put
$r=\dim_FJ(A)$.  Form the free product
\[
       A'=\overline A*F\langle y_1,\ldots,y_r\rangle.
\]
Let $I$ be the ideal of $A'$ generated by all evaluations in $A'$ of the
polynomials belonging to $\operatorname{Id}(A)$, and set $\mathcal A=A'/I.$ If $K$ is the ideal of $\mathcal A$ generated by the images of
$y_1,\ldots,y_r$, define
\begin{equation*}
       \mathcal A_{s_A}=\mathcal A/K^{s_A}.
\end{equation*} We adopt the convention \(K^0=\mathcal A\), so that \(\mathcal A_0=\{0\}\). The algebra \(\mathcal A_{s_A}\) is finite-dimensional, satisfies all polynomial identities of \(A\)
and $J(\mathcal A_{s_A})^{s_A}=0.$ In particular, when \(s_A>0\), its radical has strictly smaller
nilpotency index than \(J(A)\).

For each $i\in\{1,\ldots,q\}$, let
\begin{equation*}
       B_i=\left(\bigoplus_{j\neq i}A_j\right)
             +J(A).
\end{equation*}
The algebras $B_i$ retain the radical of $A$ but omit one simple component
of its semisimple part.

\begin{definition}
A finite-dimensional algebra $A$ is \emph{fundamental} if either $A$ is
simple or
\begin{equation}
\label{eq:fundamental-definition}
       \operatorname{Id}(A)\subsetneq \left(\bigcap_{i=1}^{q}\operatorname{Id}(B_i)\right)
       \cap\operatorname{Id}(\mathcal A_{s_A}).
\end{equation}
Fundamental algebras are also called \emph{basic algebras} in the literature.
\end{definition}

When $q=0$, $A$ is nilpotent and satisfies condition~\eqref{eq:fundamental-definition}. This condition means that the identities of $A$ cannot be recovered solely from the algebras obtained by deleting a simple component of $\overline A$ or by lowering the nilpotency index of the radical. This minimality is precisely what makes fundamental algebras the building blocks of varieties generated by finite-dimensional PI-algebras.

\begin{theorem}[{\cite[Proposition~17.2.34]{AGP2020}}] \label{thm:fundamental-decomposition}
For every finite-dimensional algebra $A$, there exist fundamental algebras
$B_1,\ldots,B_m$ such that $A\sim_T B_1\oplus\cdots\oplus B_m.$
\end{theorem}

The preceding theorem reduces PI-questions about finite-dimensional algebras
to fundamental ones.  The definition, however, is expressed through auxiliary
algebras.  The next result characterizes fundamental algebras in terms of their cocharacters. Recall that $t_A=\dim_F\overline A$ is the dimension of the semisimple part of $A$.

\begin{theorem}[{\cite[Theorem~6]{GPMZ2020}}]
\label{thm:fundamental-character}
Let $A$ be a finite-dimensional algebra over an algebraically closed field of characteristic zero with $t_A\geq1$.  Then the following
conditions are equivalent:
\begin{enumerate}
    \item $A$ is fundamental;
    \item for every sufficiently large $n$, there exists a partition
    $\lambda=(\lambda_1,\ldots,\lambda_u)\vdash n$ such that
    \[
        \lambda_{t_A+1}+\cdots+\lambda_u=s_A
        \qquad\text{and}\qquad
        m_\lambda(A)>0.
    \]
\end{enumerate}
\end{theorem}

In other words, in every sufficiently large degree the cocharacter of a
fundamental algebra contains a constituent whose Young diagram has exactly
$s_A$ boxes below its first $t_A$ rows.  This formulation is especially
effective in the polynomial-growth setting.

Recall that the algebra in~\eqref{eq:wedderburn-malcev} is called
\emph{reduced} if it is nilpotent or, after a suitable reordering of the
simple components,
\[
       A_1J(A)A_2J(A)\cdots J(A)A_q\neq \{0\}.
\]
Every fundamental algebra is reduced, see \cite[Corollary 5.13]{AKK2016}.  A reduced algebra has polynomial
codimension growth precisely when $\dim_F\overline A\leq1$.  Hence a
nonnilpotent fundamental algebra of polynomial growth has a one-dimensional
semisimple part.

\begin{proposition}[{\cite[Proposition~3.4]{GQV2026}}]
\label{prop:fundamental-polynomial-growth} Let $A$ be a finite-dimensional algebra over an algebraically closed field of characteristic zero with polynomial codimension growth.  Then $A$ is
fundamental if and only if one of the following alternatives holds:
\begin{enumerate}
    \item $A$ is nilpotent;
    \item $A\cong F$;
    \item $A\cong F+J(A)$ and, for every sufficiently large
    $n$, there exists a partition
    $\lambda=(\lambda_1,\ldots,\lambda_u)\vdash n$ such that $\lambda_2+\cdots+\lambda_u=s_A$ and $m_\lambda(A)>0.$
\end{enumerate}
\end{proposition}

Denote by $1_F$ the identity element of the fixed copy of $F$ in a Wedderburn–Malcev decomposition of $A$ above. Notice that $1_F$ need not be an identity element of $A$.

We next record the relation between the radical parameter and the precise
degree of polynomial growth.  First we establish the upper estimate, which
does not require fundamentality.

\begin{lemma}
\label{lem:polynomial-upper-bound}
Let $A$ be a finite-dimensional algebra with polynomial codimension growth.
Then there exists a constant $C>0$ such that $c_n(A)\leq Cn^{s_A}$ for every $n\geq1.$
\end{lemma}

\begin{proof}
Write $\chi_n(A)$ as in~\eqref{eq:cocharacter-decomposition}.  By Theorem~\ref{thm:bounded-colength}, every partition
$\lambda=(\lambda_1,\ldots,\lambda_u)\vdash n$ such that
$m_\lambda(A)\neq 0$ satisfies
\begin{equation*}
       n-\lambda_1=\lambda_2+\cdots+\lambda_u\leq s_A.
\end{equation*} Moreover,
Theorem~\ref{thm:bounded-colength} gives a constant $M>0$, independent of
$n$, such that $m_\lambda(A)\leq M$.

Set $k=n-\lambda_1$.  Then $0\leq k\leq s_A$ and $\lambda=(n-k,\mu)$, $\mu\vdash k.$ For fixed $k$, there are at most $p(k)$ possible partitions of this form,
where $p(k)$ denotes the number of partitions of $k$.  By the hook-length
formula,
\[
       d_\lambda
       =\frac{n!}{\prod_{(i,j)\in[\lambda]}h_{ij}}
       \leq\frac{n!}{\lambda_1!}
       =\frac{n!}{(n-k)!}
       \leq n^k
       \leq n^{s_A}.
\]
It follows that
$$
       c_n(A)
       \leq M\sum_{k=0}^{s_A}p(k)n^k 
       \leq M\left(\sum_{k=0}^{s_A}p(k)\right)n^{s_A}.$$
The conclusion follows by taking
$C=M\sum_{k=0}^{s_A}p(k)$.
\end{proof}

\begin{theorem}
\label{thm:fundamental-growth-degree}
Let $A$ be a finite-dimensional, nonnilpotent fundamental algebra with
polynomial codimension growth.  Then there exist constants $c,C>0$ and
$n_0\geq1$ such that
\[
       cn^{s_A}\leq c_n(A)\leq Cn^{s_A}
       \qquad\text{for every }n\geq n_0.
\]
Consequently, $s_A$ is exactly the degree of the polynomial growth of the
codimension sequence of $A$.
\end{theorem}

\begin{proof}
If $A\cong F$, then $s_A=0$ and $c_n(A)=1$ for every $n$, so the
conclusion is immediate. We may therefore assume that $A\not\cong F$.
The upper estimate follows from Lemma~\ref{lem:polynomial-upper-bound}. For the
lower estimate, Proposition~\ref{prop:fundamental-polynomial-growth} implies
that, for every sufficiently large $n$, there exists a partition
$\lambda\vdash n$ such that
\[
  \lambda_2+\cdots+\lambda_u=s_A
  \quad\text{and}\quad
  m_\lambda(A)>0.
\] Set \(s=s_A\) and \(\mu=(\lambda_2,\ldots,\lambda_u)\). Then
\(\mu\vdash s\) and \(\lambda=(n-s,\mu)\). Since \(n\) is sufficiently large, after increasing the threshold if necessary, we may assume that \(n\geq2s\), and hence \(n-s\geq s\).

We first establish a uniform estimate for all partitions $\mu\vdash s$.
Let $\mu'$ denote the conjugate partition of $\mu$. The hook lengths of
the boxes below the first row of $\lambda=(n-s,\mu)$ coincide with those of
$\mu$ and therefore have product $s!/d_\mu$. Moreover, the hook length
of the box in position $(1,j)$ is
$ n-s-j+1+\mu'_j.$ Consequently, the product of the hook lengths in the first row is
\[    \prod_{j=1}^{n-s}h_{1j}
    =(n-s)!\,R_{n,\mu},\quad \mbox{ where }\quad  R_{n,\mu}
      =
      \prod_{j=1}^{\mu_1}
      \left(1+\frac{\mu'_j}{n-s-j+1}\right).
\]
It follows that
\[
    \prod_{(ij)\in[\lambda]}h_{ij}
      =\frac{s!}{d_\mu}(n-s)!\,R_{n,\mu}.
\]

For every fixed $\mu\vdash s$, we have $R_{n,\mu}\to1$ as
$n\to\infty$. Since there are only finitely many partitions of $s$,
there exists $n_0$ such that $R_{n,\mu}\leq2$ for every $\mu\vdash s$ and every $n\geq n_0$. Hence, by the
hook-length formula,
\[
\begin{aligned}
    d_{(n-s,\mu)}=\frac{d_\mu}{s!R_{n,\mu}}
        \frac{n!}{(n-s)!}\geq\frac{d_\mu}{2s!}
        n(n-1)\cdots(n-s+1).
\end{aligned}
\]
Each factor in the product above is then at least $n/2$, and therefore
\[
    d_{(n-s,\mu)}
      \geq\frac{d_\mu}{2^{s+1}s!}\,n^s
      \geq\frac{1}{2^{s+1}s!}\,n^s.
\]
Thus, setting $c=\frac{1}{2^{s+1}s!},$ we obtain $d_{(n-s,\mu)}\geq cn^s$ for every $\mu\vdash s$ and every sufficiently large $n$.

Finally, for each sufficiently large $n$, choose the partition
$\lambda=(n-s,\mu)$ supplied by
Proposition~\ref{prop:fundamental-polynomial-growth}. Since
$m_\lambda(A)>0$, its multiplicity is at least one, and hence
\[
    c_n(A)
      \geq m_\lambda(A)d_\lambda
      \geq d_\lambda
      \geq cn^{s_A}.
\]
This proves the lower estimate.

\end{proof}

\begin{corollary}
\label{cor:quadratic-growth-criterion}
Let $A$ be a finite-dimensional, nonnilpotent fundamental algebra with
polynomial codimension growth.  Then $A$ has quadratic codimension growth if
and only if $s_A=2$.
\end{corollary}

\begin{proof}
This follows immediately from
Theorem~\ref{thm:fundamental-growth-degree}. More explicitly, when $s_A=2$,
the partitions supplied by Theorem~\ref{thm:fundamental-character} are $(n-2,2)$ and $(n-2,1,1)$, whose degrees are, respectively,
\[
       d_{(n-2,2)}=\frac{n(n-3)}2,
       \qquad
       d_{(n-2,1,1)}=\frac{(n-1)(n-2)}2.
\]
At least one of them occurs with non-zero multiplicity for every sufficiently
large $n$, giving the quadratic lower bound; the quadratic upper bound follows
from Lemma~\ref{lem:polynomial-upper-bound}.  
\end{proof}

Theorem~\ref{thm:fundamental-decomposition} and
Corollary~\ref{cor:quadratic-growth-criterion} identify the structural core of
the classification problem.  After passing to finite dimension, the
nonnilpotent components that can contribute genuinely quadratic growth are
precisely the fundamental algebras with one-dimensional semisimple part and
$s_A=2$.  The following sections introduce the concrete algebras that generate
the minimal varieties of quadratic growth and use them to describe these
fundamental components up to PI-equivalence.

\section{Fundamental algebras with at most quadratic codimension growth}

We recall some algebras with at most quadratic codimension growth, together with their cocharacters and codimensions. Let $UT_m=UT_m(F)$ denote the algebra of $m\times m$ upper triangular matrices over $F$, and let $e_{ij}$ denote the usual matrix units.

We first consider the subalgebras of $UT_2:$
\[
   A_1=Fe_{11}+Fe_{12}
   \qquad\text{and}\qquad
   A_1^*=Fe_{22}+Fe_{12}.
\]

\begin{lemma}[{\cite[Lemma~3]{GLM2005}}]
\label{lem:A1-A1star}
For every $n\geq2$, we have $\chi_n(A_1)=\chi_n(A_1^*)
   =\chi_{(n)}+\chi_{(n-1,1)}$ and    $c_n(A_1)=c_n(A_1^*)=n$.
\end{lemma}

Next, consider the following unitary subalgebra of $UT_3$:
\[
   A_2=F(e_{11}+e_{22}+e_{33})+Fe_{12}+Fe_{13}+Fe_{23}.
\]

\begin{lemma}[{\cite[Lemma~4]{GLM2005}}]
\label{lem:A2}
For every $n\geq4$, we have   $\chi_n(A_2)
   =\chi_{(n)}+\chi_{(n-1,1)}+\chi_{(n-2,1,1)}$ and $
   c_n(A_2)=\frac{n(n-1)+2}{2}.$
\end{lemma}

Finally, consider the following subalgebras of $UT_3$:
 $$  A_4   =Fe_{11}+Fe_{12}+Fe_{13}+Fe_{23},\quad 
   A_4^* =Fe_{33}+Fe_{12}+Fe_{13}+Fe_{23},$$
$$A_5=Fe_{22}+Fe_{12}+Fe_{13}+Fe_{23}\quad \mbox{ and }\quad A_6=F(e_{11}+e_{33})+Fe_{12}+Fe_{13}+Fe_{23}.$$

\begin{lemma}[{\cite[Lemma~6]{GLM2005}, \cite[Theorem~3.1]{JorgeVieira2006}}]
\label{lem:A4-A5-A6}
For every $B\in\{A_4,A_4^*,A_5,A_6\}$ and every $n\geq4$, we have $\chi_n(B)
   =\chi_{(n)}+2\chi_{(n-1,1)}
    +\chi_{(n-2,2)}+\chi_{(n-2,1,1)}$ and $c_n(B)=n(n-1)$.
\end{lemma}

Set
\[
   \mathcal L=\{A_1,A_1^*\}
   \qquad\text{and}\qquad
   \mathcal Q=\{A_2,A_4,A_4^*,A_5,A_6\}.
\]
The preceding lemmas show that the algebras in $\mathcal L$ have linear codimension growth, whereas those in $\mathcal Q$ have quadratic codimension growth. Recall that a variety $\mathcal V$ is said to be \emph{minimal of quadratic growth} if its codimension sequence has quadratic growth, whereas the
codimension sequence of every proper subvariety of $\mathcal V$ has at most linear growth. The following theorem shows that the algebras above exhaust all possibilities.

\begin{theorem}[{\cite[Corollary~4.6]{JorgeVieira2006}}]
\label{thm:minimal-quadratic-varieties}
The minimal varieties of quadratic codimension growth are precisely those generated by one of the algebras $A_2$, $A_4$, $A_4^*$, $A_5$ or $A_6$.
\end{theorem}

Recall that every algebra in $\mathcal L\cup\mathcal Q$ has a one-dimensional semisimple part. Their cocharacter decompositions allow us to apply Theorem~\ref{thm:fundamental-character} and obtain the following result.

\begin{proposition}
\label{prop:quadratic-algebras-fundamental}
Every algebra $M\in\mathcal L\cup\mathcal Q$ is fundamental. Moreover, $s_A=1$ and $s_B=2$ for all $A\in\mathcal L$ and $B\in\mathcal Q$.
\end{proposition}

\begin{proof}
For every $B\in\mathcal Q$, we have $J(B)=Fe_{12}+Fe_{13}+Fe_{23},$ with $J(B)^2=Fe_{13}\neq\{0\}$ and $J(B)^3=\{0\}$. Hence $s_B=2$. By Lemmas~\ref{lem:A2} and~\ref{lem:A4-A5-A6}, the constituent $\chi_{(n-2,1,1)}$ occurs with positive multiplicity in $\chi_n(B)$ for every $B\in\mathcal Q$ and every $n\geq4$. The corresponding partition $\lambda=(n-2,1,1)$ satisfies $\lambda_2+\lambda_3=2=s_B$. Since $t_B=1$, Theorem~\ref{thm:fundamental-character} implies that every $B\in\mathcal Q$ is fundamental.

The same argument applies to every $A\in\mathcal L$, with $s_A=1$, since $J(A)=Fe_{12}\neq\{0\}$, $J(A)^2=\{0\}$ and $\chi_{(n-1,1)}$ occurs in $\chi_n(A)$ for every $n\geq2$.
\end{proof}

Thus, the algebras in $\mathcal Q$ play two complementary roles: they generate precisely the minimal varieties of quadratic growth and, at the same time, they are fundamental algebras with radical parameter equal to $2$. 

\section{Decomposition of the Jacobson radical and the detector algebra}

In this section, we consider finite-dimensional algebras $A=F+J$, where $J=J(A)$ and $J^3=\{0\}$. This setting is motivated by Proposition~\ref{prop:fundamental-polynomial-growth} and Corollary~\ref{cor:quadratic-growth-criterion}, which together show that every nonnilpotent fundamental algebra of quadratic codimension growth has this form. To determine the multilinear identities of such an algebra, it suffices to consider evaluations with at most two radical entries, all remaining entries being $1_F$. We use the components of $J$ to construct a direct sum $B(A)$ of the algebras introduced in Section~3 and prove that $A\sim_T B(A)$ whenever $A$ is fundamental and has quadratic codimension growth.

By \cite[Lemma~9.7.1]{GiambrunoZai}, we have the direct vector-space
decomposition
\begin{equation} \label{decompj}
    J(A)=J_{10}+J_{01}+J_{00}+J_{11},
\end{equation}
where $J_{ik}
=
\{a\in J\mid 1_Fa=ia,\ a1_F=ka\},$ $ i,k\in\{0,1\}.$ Moreover, $J_{ik}J_{rs}\subseteq\delta_{kr}J_{is}$ where $\delta_{kr}$ is the Kronecker delta function.

We introduce the following subspaces, defined in terms of products and commutators of the components above:
$$
 P_4:=J_{10}J_{00}\subseteq J_{10},\quad P_{4^*}:=J_{00}J_{01}\subseteq J_{01},$$ $$
 P_5:=J_{01}J_{10}\subseteq J_{00},
 \quad P_6:=J_{10}J_{01}\subseteq J_{11}\quad \mbox{ and }\quad  
 C:=[J_{11},J_{11}]\subseteq J_{11}^2.
$$

\begin{remark}\label{lem:ideais}
Let $A=F+J$ be an algebra with $J=J(A)$ and suppose that $J^3=\{0\}$. Then $P_4,P_{4^*},P_5,P_6$ and $C$ are
two-sided ideals of $A$. Moreover,
$P_5\cap P_6=\{0\}.$
\end{remark}

\begin{proof}
Let $I$ be any of the subspaces $P_4$, $P_{4^*}$, $P_5$, $P_6$ and $C$. By definition, every element of $I$ is a linear combination of products of two elements of $J$. Thus
$I\subseteq J^2$, and the hypothesis $J^3=\{0\}$ gives
$JI=IJ=\{0\}$. Moreover, the multiplication rules for the components
in \eqref{decompj} yield
\[
   P_4\subseteq J_{10},\qquad
   P_{4^*}\subseteq J_{01},\qquad
   P_5\subseteq J_{00},\qquad
   P_6,C\subseteq J_{11}.
\]
For $x\in J_{ij}$, we have $1_Fx=ix$ and $x1_F=jx$.
Hence $I$ is stable under left and right multiplication
by $1_F$. Since $A=F+J$, it follows that $I$ is a
two-sided ideal of $A$.

Finally, $P_5\subseteq J_{00}$ and $P_6\subseteq J_{11}$,
while $J_{00}\cap J_{11}=\{0\}$ by the directness of
\eqref{decompj}. Therefore, $P_5\cap P_6=\{0\}$.
\end{proof}

We next recall some results from \cite{GLM2005}. Throughout the next two lemmas, we fix a Wedderburn--Malcev decomposition $A=F+J(A)$ of a finite-dimensional algebra.

\begin{lemma}[{\cite[Lemmas~9, 10, and~15]{GLM2005}}]
\label{A1A2A4}
The following statements hold:
\begin{enumerate}
\item If \(J_{10}\neq\{0\}\), then \(A_1\in\operatorname{var}(A)\).
      If \(J_{01}\neq\{0\}\), then \(A_1^*\in\operatorname{var}(A)\).
\item If \(C=[J_{11},J_{11}]\neq\{0\}\), then
      \(A_2\in\operatorname{var}(A)\).
\item If \(P_4=J_{10}J_{00}\neq\{0\}\), then
      \(A_4\in\operatorname{var}(A)\). If
      \(P_4^*=J_{00}J_{01}\neq\{0\}\), then
      \(A_4^*\in\operatorname{var}(A)\).
\end{enumerate}
\end{lemma}

The next lemma provides
analogous criteria for $A_5$ and $A_6$ to belong to $\operatorname{var}(A)$.

\begin{lemma} \label{lem:A5A6-detectors}
Let $A=F+J$, where $J=J(A)$, and suppose that $J^3=\{0\}$.
\begin{enumerate}
\item If $P_5=J_{01}J_{10}\neq\{0\}$, then
      $A_5\in\operatorname{var}(A)$.
\item If $P_6=J_{10}J_{01}\neq\{0\}$, then
      $A_6\in\operatorname{var}(A)$.
\end{enumerate}
\end{lemma}

\begin{proof}
By Remark~\ref{lem:ideais}, $P_5$ and $P_6$ are ideals with trivial
intersection.

Suppose first that $P_5\neq \{0\}$, and consider the quotient
$\overline A=A/P_6$. The product $J_{01}J_{10}$ remains nonzero, whereas
$J_{10}J_{01}=\{0\}$ in $\overline A$. Choose $a\in J_{01}$ and
$b\in J_{10}$ such that $ab\neq0$. The subalgebra generated by $1_F,a,b$
has basis $1_F,a,b,ab$ and satisfies
\[
 1_Fa=0,\quad a1_F=a,\quad 1_Fb=b,\quad b1_F=0,
 \quad ab\neq0,\quad ba=0.
\]
Since $J^3=\{0\}$, all other products of radical elements vanish. This
subalgebra is isomorphic to $A_5$. Therefore, $A_5\in \textnormal{var}(\overline{A})\subseteq \textnormal{var}(A).$

Suppose now that $P_6\neq \{0\}$, and consider the quotient $A/P_5$. Choose
$a\in J_{10}$ and $b\in J_{01}$ such that $ab\neq0$. Then
\[
 1_Fa=a,\quad a1_F=0,\quad 1_Fb=0,\quad b1_F=b,
 \quad ab\neq0,\quad ba=0,
\]
and $1_F(ab)=(ab)1_F=ab$. The subalgebra generated by $1_F,a,b$ is isomorphic
to $A_6$. Therefore, $A_6\in \textnormal{var}(A/P_5)\subseteq \textnormal{var}(A).$
\end{proof}

Let $A=F+ J$ such that $J^3=\{0\}$. For any subspace $U\subseteq A$, define
\[
 \delta(U)=
 \begin{cases}
  1, & \text{if } U\neq \{0\},\\
  0, & \text{if } U=\{0\}.
 \end{cases}
\]
We consider the following indicator parameters:
$$
 \varepsilon_1     =\delta(J_{10}),\quad 
 \varepsilon_{1^*}=\delta(J_{01}),\quad 
 \varepsilon_2    =\delta(C),$$ $$
 \varepsilon_4     =\delta(P_4),\quad 
 \varepsilon_{4^*}=\delta(P_{4^*}),\quad 
 \varepsilon_5    =\delta(P_5),\quad 
 \varepsilon_6    =\delta(P_6).
$$

If $M$ is an algebra and $\eps\in\{0,1\}$, the notation
$M^\eps$ means that $M$ occurs as a direct summand when $\eps=1$ and is
omitted when $\eps=0$. Set
\begin{equation}
\label{eq:detector-sum}
 B(A)=A_1^{\eps_1}\oplus (A_1^*)^{\eps_{1^*}}
 \oplus A_2^{\eps_2}\oplus A_4^{\eps_4}
 \oplus (A_4^*)^{\eps_{4^*}}\oplus A_5^{\eps_5}
 \oplus A_6^{\eps_6}.
\end{equation} Thus $B(A)$ records the nonvanishing of $J_{10}$ and $J_{01}$, together with the selected products and commutators that can produce quadratic growth.  

\begin{remark}\label{rem:B(A)}
Notice that, by Lemmas~\ref{A1A2A4} and~\ref{lem:A5A6-detectors}, every nonzero summand of \(B(A)\) lies in \(\operatorname{var}(A)\). Hence $B(A)\in\operatorname{var}(A)$, and therefore $\operatorname{Id}(A)\subseteq\operatorname{Id}(B(A)).$
\end{remark}

We first show that, in the absence of all quadratic detectors, the
codimension sequence cannot have quadratic growth.

\begin{lemma}
\label{lem:no-quadratic-mechanism}
Let $A=F+J(A)$ be a finite-dimensional algebra with $J(A)^3=\{0\}$. If $C=P_4=P_{4^*}=P_5=P_6=\{0\}$ then there exists a constant $\alpha> 0$ such that $c_n(A)\leq  \alpha n$, for all $n \geq 1$. Consequently, a fundamental algebra with $s_A=2$
cannot satisfy these five equalities simultaneously.
\end{lemma}

\begin{proof}
Set $R=F+ J_{01}+ J_{10}+ J_{11}$ and $N=J_{00}.$ We first show that
$A=R\oplus N$ as a direct sum of algebras.

By the multiplication rule, $J_{ik}J_{rs}\subseteq\delta_{kr}J_{is}.$ Therefore,
\[
 J_{01}J_{00}=J_{11}J_{00}
 =J_{00}J_{10}=J_{00}J_{11}=\{0\}.
\]
Moreover, the assumptions $P_4=J_{10}J_{00}=\{0\}$ and $P_{4^*}=J_{00}J_{01}=\{0\}$ give the remaining mixed products involving $J_{00}$. Since $1_FJ_{00}=J_{00}1_F=\{0\},$ we obtain $RN=NR=\{0\}.$ Thus $N$ is a two-sided ideal of $A$ and $A=R\oplus N.$ Since $N\subseteq J$ and $J^3=\{0\}$, the algebra $N$ is nilpotent, with
$N^3=\{0\}$.

The equality $C=[J_{11},J_{11}]=\{0\}$ shows that $J_{11}$ is commutative. Furthermore, $P_5=J_{01}J_{10}=\{0\}$ and $ P_6=J_{10}J_{01}=\{0\}.$ We distinguish four cases.

If $J_{10}=J_{01}=
\{0\},$ then $R=F+ J_{11}$
is commutative. Thus, $\Id(R)=\Id(F),$ and so $R\simeqT F.$

Suppose that $J_{10}\neq \{0\}$ and $
 J_{01}=\{0\}.$ Then  $R=F+ J_{10}+ J_{11}.$ Since $J_{11}$ is commutative,
\cite[Lemma 10]{GLM2005} gives $R\simeqT A_1.$ 

Similarly, if $J_{10}=\{0\}$ and $
 J_{01}\neq \{0\},$ then \cite[Lemma 10]{GLM2005} yields $R\simeqT A_1^*.$

Finally, assume that $J_{10}\neq\{0\}$ and $J_{01}\neq \{0\}.$
Since $J_{11}$ is commutative and
$J_{10}J_{01}=J_{01}J_{10}=\{0\},$ \cite[Lemma 16]{GLM2005} implies
$R\simeqT A_1^*\oplus A_1.$

Since $A=R\oplus N$, we conclude that $A$ is PI-equivalent to one of
\[
 F\oplus N,\qquad
 A_1\oplus N,\qquad
 A_1^*\oplus N,\qquad
 A_1\oplus A_1^*\oplus N.
\]
These are precisely the nonnilpotent algebras occurring in the linear
classification of Giambruno and La Mattina
\cite[Theorem~22]{GLM2005}. Consequently, $c_n(A)\leq  \alpha n$, for all $n \geq 1$.

If $A$ is fundamental with $s_A=2$, Corollary \ref{cor:quadratic-growth-criterion} gives
$c_n(A)\geq \gamma n^2$ for some $\gamma>0$ and all sufficiently large $n$. This contradicts $c_n(A)\leq  \alpha n$, $n\geq 1$. Therefore, the consequence follows.
\end{proof}

The next result is the main technical step.

\begin{lemma}
\label{lem:transfer}
Let $A=F+J$ satisfying $J^3=\{0\}$ and suppose that at least one of the subspaces
$C$, $P_4$, $P_{4^*}$, $P_5$, $P_6$ is nonzero. Then
$\Id(B(A))\subseteq\Id(A)$ and $\eps_2+\eps_4+\eps_{4^*}+\eps_5+\eps_6\geq1.$
\end{lemma}

\begin{proof}
We first note that, by hypothesis, at least one of the subspaces
$C,P_4,P_4^*,P_5$, and $P_6$ is nonzero. Therefore, by the definition of the parameters $\varepsilon_i$ we have $\varepsilon_2+\varepsilon_4+\varepsilon_{4^*} +\varepsilon_5+\varepsilon_6\geq1.$

Since the ground field has characteristic zero, the $T$-ideal of an
algebra is determined by its multilinear elements. It is therefore
enough to prove that
$P_n\cap\Id(B(A))\subseteq P_n\cap\Id(A)$ for every $n\geq1$. Let
\[
 f(x_1,\ldots,x_n)
 =
 \sum_{\sigma\in S_n}
 \lambda_\sigma
 x_{\sigma(1)}\cdots x_{\sigma(n)}
 \in P_n\cap\Id(B(A)).
\]
Fix a vector-space basis of $A$ consisting of $1_F$ together with
homogeneous bases of
$J_{00},$ $ J_{01}$, $J_{10}$ and $J_{11}.$
By multilinearity, it suffices to show that every evaluation of $f$ on
this basis is zero.

If an evaluation contains $r$ radical entries then the value of every
monomial occurring in $f$ belongs to $J^r$. Consequently, every evaluation
containing at least three radical entries is zero since $J^3=\{0\}$.

Since $B(A)$ contains at least one of the algebras
$A_2,A_4,A_4^*,A_5,A_6$, each of which contains a copy of $F$, we have
$F\in\operatorname{var}(B(A))$. Therefore,
\begin{equation}
\label{eq:IdBA-IdF}
   \operatorname{Id}(B(A))\subseteq\operatorname{Id}(F).
\end{equation}

We now consider separately the evaluations containing zero, one or
two radical values.

If every variable is evaluated at $1_F$, then
\[
 f(1_F,\ldots,1_F)
 =
 \left(\sum_{\sigma\in S_n}\lambda_\sigma\right)1_F.
\]
On the other hand, \eqref{eq:IdBA-IdF} implies that $f$ is an identity
of $F$. Therefore, $\sum_{\sigma\in S_n}\lambda_\sigma=0.$
Hence $f(1_F,\ldots,1_F)=0$.

Suppose that $x_p$ is evaluated at an element
$a\in J_{ik}$ and all the remaining variables are evaluated at $1_F$.

If $a\in J_{10}$ then
$1_Fa=a$ and $a1_F=0.$ Thus a monomial has nonzero value precisely when $x_p$ is its last
variable, and every such monomial evaluates to $a$. Since
$J_{10}\neq \{0\}$, by Lemma \ref{A1A2A4}, the algebra $A_1$ occurs as a summand of $B(A)$.
Evaluating $x_p$ at $e_{12}$ and all remaining variables at $e_{11}$
in $A_1$, exactly the same monomials survive, with the same
coefficients. Since $f\in\Id(A_1)$, their coefficient sum is zero.
Therefore the original evaluation on $A$ is zero.

If $a\in J_{01}$, then
$1_Fa=0$ and $a1_F=a.$ In this case a monomial is nonzero precisely when $x_p$ is its first
variable. Since
$J_{01}\neq \{0\}$, by Lemma \ref{A1A2A4}, the algebra $A_1^*$ occurs as a summand of $B(A)$. The same coefficient sum is obtained by evaluating $x_p$ at
$e_{12}$ and all remaining variables at $e_{22}$ in $A_1^*$.
Since $f\in\Id(A_1^*)$, the evaluation is zero.

If $a\in J_{11}$ then $1_Fa=a1_F=a$. Consequently, by \eqref{eq:IdBA-IdF},
\[
 f(1_F,\ldots,a,\ldots,1_F)
 =
 \left(\sum_{\sigma\in S_n}\lambda_\sigma\right)a
 =
 0.
\]

Finally, if $a\in J_{00}$ and $n\geq2$ then every monomial contains an
occurrence of $1_F$ on the left or on the right of $a$. Since $1_Fa=a1_F=0,$
the evaluation is zero. If $n=1$ then $f=\lambda x_1$. Since $B(A)$
is nonzero, $f\in\Id(B(A))$ implies $\lambda=0$.

Now we assume that the evaluation contains two radical entries. Fix distinct variables $x_p,x_q$ and evaluate them at
elements $a\in J_{ik}$ and $b\in J_{rs},$
respectively. Evaluate every other variable at $1_F$.

For each $\sigma\in S_n$, put
\[
 m_\sigma=x_{\sigma(1)}\cdots x_{\sigma(n)}.
\]
Since the left and right actions of $1_F$ on every $J_{i_1i_2}$ are
either zero or the identity, there exist numbers $\eta_\sigma^{ab},\eta_\sigma^{ba}\in\{0,1\},$ depending only on $\sigma$, on $p,q$, and on the component $J_{uv}$ of $a$
and $b$, such that
\[
 m_\sigma(1_F,\ldots,a,\ldots,b,\ldots,1_F)
 =
 \eta_\sigma^{ab}ab+\eta_\sigma^{ba}ba.
\]
Here $\eta_\sigma^{ab}=1$ precisely when $x_p$ occurs before $x_q$ in
$m_\sigma$ and all occurrences of $1_F$ are compatible with the
indices of $J_{i_1i_2}$. Similarly, the coefficient $\eta_\sigma^{ba}$ is defined as $\eta_\sigma^{ba}=1$ precisely when $x_q$ occurs before $x_p$ in
$m_\sigma$ and all occurrences of $1_F$ are compatible with the nonzero products. In particular, these numbers are defined independently of whether
$ab$ or $ba$ is zero.

Set $\alpha=
 \sum_{\sigma\in S_n}
 \lambda_\sigma\eta_\sigma^{ab}$ and $
 \beta=
 \sum_{\sigma\in S_n}
 \lambda_\sigma\eta_\sigma^{ba}.$ Then
\begin{equation}
\label{eq:two-radicals-correct}
 f(1_F,\ldots,a,\ldots,b,\ldots,1_F)
 =
 \alpha ab+\beta ba.
\end{equation}

We now examine all products allowed by the multiplication rule $J_{ik}J_{rs}\subseteq\delta_{kr}J_{is}.$

Assume first that \(a,b\in J_{11}\). In this case every occurrence of $1_F$ acts as the identity, and hence
every monomial evaluates either to $ab$ or to $ba$. Recall, by  \eqref{eq:IdBA-IdF}, that, in this case, $\alpha+\beta= \sum_{\sigma\in S_n}\lambda_\sigma=0$. Thus,
\[
 \alpha ab+\beta ba
 =
 \alpha(ab-ba)
 =
 \alpha[a,b].
\]

If $[a,b]=0$, the evaluation is zero. Suppose that $[a,b]\neq0$. Then
$C=[J_{11},J_{11}]\neq \{0\}$ and by Lemma \ref{A1A2A4} $A_2$ occurs in $B(A)$ and $f\in\Id(A_2)$. Evaluate $x_p$ at
$e_{12}$, $x_q$ at $e_{23}$, and every remaining variable at $I_3=e_{11}+e_{22}+e_{33}$.
The monomials in which $x_p$ precedes $x_q$ yield $e_{13}$, whereas
the monomials in the opposite order vanish. Thus
\[
 0
 =
 f(I_3,\ldots,e_{12},\ldots,e_{23},\ldots,I_3)
 =
 \alpha e_{13}.
\]
Therefore $\alpha=0$, and the original evaluation is zero.

Suppose now that \(a\in J_{10}\) and \(b\in J_{00}\). The multiplication rule gives
$ab\in J_{10}$ and $ba=0$. If $ab=0$, there is nothing to prove. Assume that $ab\neq0$. Then
$P_4\neq \{0\}$, by Lemma \ref{A1A2A4}, $A_4$ occurs in $B(A)$ and $f\in\Id(A_4)$. Let $e_4=e_{11}$, $u=e_{12}$, $v=e_{23}$ in $A_4$. Relative to $e_4$, one has $u\in J(A_4)_{10}$, $v\in J(A_4)_{00}$,
$uv=e_{13}$ and $vu=0.$ Since components $J_{i_1\, i_2}$ of $u,v$ agree with those of $a,b$, respectively,
the same monomials contribute to $uv$ as contribute to $ab$ in
\eqref{eq:two-radicals-correct}. Hence
\[
 0
 =
 f(e_4,\ldots,u,\ldots,v,\ldots,e_4)
 =
 \alpha e_{13}.
\]
Thus $\alpha=0$.

Assume that \(a\in J_{00}\) and \(b\in J_{01}\). Here $ab\in J_{01}$ and $ba=0$. If $ab\neq0$ then $P_{4^*}\neq \{0\}$, and therefore, by Lemma \ref{A1A2A4}, $A_4^*$ occurs in $B(A)$ and $f\in\Id(A_4^*)$. Take $e_{4^*}=e_{33}$, $u=e_{12}$ and $v=e_{23}$ in $A_4^*$. Then
$u\in J(A_4^*)_{00}$, $v\in J(A_4^*)_{01},$
$uv=e_{13}$ and $vu=0$. The same coefficient comparison gives
\[
 0
 =
 f(e_{4^*},\ldots,u,\ldots,v,\ldots,e_{4^*})
 =
 \alpha e_{13},
\]
and hence $\alpha=0$.

Consider now the case in which \(a\in J_{10}\) and \(b\in J_{01}\).
This is the only case involving distinct components in which both oriented products may be nonzero.
Indeed, $ab\in J_{11}$ and $ ba\in J_{00}.$
Since the decomposition is direct, $J_{11}\cap J_{00}=\{0\}.$

Suppose first that $ab\neq0$. Then $P_6\neq \{0\}$, and by Lemma \ref{lem:A5A6-detectors}, $A_6$ occurs in
$B(A)$. In $A_6$, let $e_6=e_{11}+e_{33}$,
$u=e_{12}\in J(A_6)_{10}$ and $ v=e_{23}\in J(A_6)_{01}.$ One has $uv=e_{13}$ and $vu=0$. Therefore
\[
 0
 =
 f(e_6,\ldots,u,\ldots,v,\ldots,e_6)
 =
 \alpha e_{13},
\]
and hence $\alpha=0$.

Suppose now that $ba\neq0$. Then $P_5\neq \{0\}$, so
$f\in\Id(A_5)$. In $A_5$, let
$e_5=e_{22}$, $u=e_{23}\in J(A_5)_{10}$ and $ v=e_{12}\in J(A_5)_{01}$. In this case $uv=0$ and $vu=e_{13}.$ Consequently,
\[
 0
 =
 f(e_5,\ldots,u,\ldots,v,\ldots,e_5)
 =
 \beta e_{13},
\]
so $\beta=0$. Thus both contributions in
\eqref{eq:two-radicals-correct} vanish separately. 

Assume that \(a\in J_{11}\) and \(b\in J_{10}\). The multiplication rule gives $ab\in J_{10}$ and $ba=0.$ If $ab=0$, the evaluation is zero. Suppose that $ab\neq0$. A monomial
contributes to the coefficient of $ab$ precisely when the variable
evaluated at $b$ is the last variable of that monomial. Indeed, $1_F$
acts as the identity on both sides of $a$ and on the left of $b$, while
$b1_F=0$.

 Since $J_{10}\neq \{0\}$, the algebra $A_1$ occurs in $B(A)$,
and therefore $f\in\Id(A_1)$. Now evaluate the variable corresponding to $a$ at $e_{11}$, the
variable corresponding to $b$ at $e_{12}$, and all remaining variables
at $e_{11}$ in $A_1$. Exactly the same monomials survive, with the same
coefficients. Hence
$\alpha e_{12}=0,$ so $\alpha=0$.

Consider now \(a\in J_{01}\) and \(b\in J_{11}\). Here $ab\in J_{01}$ and $ba=0.$
A monomial contributes to $ab$ precisely when the variable evaluated
at $a$ is its first variable. In this case, we recall that $A_1^*$ occurs as a direct summand of $B(A)$. Replacing $a$ by $e_{12}$, $b$ by
$e_{22}$, and all remaining entries by $e_{22}$ in $A_1^*$ yields the
same coefficient. Since $f\in\Id(A_1^*)$, this coefficient is zero.

Finally, assume that \(a,b\in J_{00}\).
If $n\geq3$, each monomial contains at least one occurrence of $1_F$.
That occurrence lies either before both radical entries, between them
or after both. Since $1_F$ annihilates $J_{00}$ on both sides, every
monomial evaluates to zero.

Suppose that $n=2$. Then
$f(x_1,x_2)=\gamma x_1x_2+\delta x_2x_1.$ Choose a quadratic model $Q$ occurring in $B(A)$. In each of the
algebras $A_2$, $A_4$, $A_4^*$, $A_5$ and $A_6$, the elements $e_{12},e_{23}$ satisfy
$e_{12}e_{23}=e_{13}\neq0$ and $e_{23}e_{12}=0.$
Since $f\in\Id(Q)$, computing first
$f(e_{12},e_{23})$ and then $f(e_{23},e_{12})$ gives
$\gamma e_{13}=0$ and $
 \delta e_{13}=0.$ Thus $\gamma=\delta=0$.

The eight potentially nonzero products are
\[
 J_{00}J_{00},\quad
 J_{00}J_{01},\quad
 J_{01}J_{10},\quad
 J_{01}J_{11},\quad
 J_{10}J_{00},\quad
 J_{10}J_{01},\quad
 J_{11}J_{10},\quad
 J_{11}J_{11}.
\]
All of them have been considered above, up to interchanging $a$ and
$b$. Every other product vanishes by the multiplication rule.
Therefore every basis evaluation of $f$ on $A$ is zero and hence
$f\in\Id(A)$.

We have proved $P_n\cap\Id(B(A))
 \subseteq
 P_n\cap\Id(A)$ for every $n\geq1$. Since the ground field has characteristic zero,
multilinearization yields $\Id(B(A))\subseteq\Id(A).$ \end{proof}

\section{Fundamental algebras with quadratic codimension growth}
\label{sec:classification}

In this section, we classify, up to PI-equivalence, the fundamental algebras $A$ satisfying $s_A=2$. This enables us to prove that every variety of quadratic codimension growth is generated by a finite direct sum of algebras whose corresponding varieties are minimal and have at most quadratic growth. Recall that the minimal varieties of linear and quadratic growth were classified in \cite[Theorem 22]{GLM2005} and {\cite[Corollary~4.6]{JorgeVieira2006}}, respectively.

To state the classification, recall the detector algebra associated with
$A$:
\begin{equation*}
 B(A)=A_1^{\eps_1}\oplus (A_1^*)^{\eps_{1^*}}
 \oplus A_2^{\eps_2}\oplus A_4^{\eps_4}
 \oplus (A_4^*)^{\eps_{4^*}}\oplus A_5^{\eps_5}
 \oplus A_6^{\eps_6}.
\end{equation*}

Here, each parameter $\varepsilon_i\in\{0,1\}$ records the nonvanishing of a relevant component in the decomposition~\eqref{decompj} of $J(A)$, or of one of the product or commutator subspaces determined by these components. Accordingly, $B(A)$ is the direct sum of precisely those model algebras $A_i$ whose associated detectors are nonzero. Each selected algebra $A_i$ belongs to $\operatorname{var}(A)$ and generates a minimal subvariety of $\operatorname{var}(A)$.

\begin{theorem}\label{thm:main}
Let $F$ be an algebraically closed field of characteristic zero and let $A$ be a finite-dimensional, nonnilpotent fundamental $F$-algebra with polynomial codimension growth. Assume that
$s_A=2$. Then $A\simeqT B(A).$ Equivalently, there exist $\eps_1$, $\eps_{1^*}$, $\eps_2$, $\eps_4$, $\eps_{4^*}$, $\eps_5$, $\eps_6\in\{0,1\}$ such that
\begin{equation*}
\label{eq:main-classification}
 A\simeqT
 A_1^{\eps_1}\oplus(A_1^*)^{\eps_{1^*}}
 \oplus A_2^{\eps_2}\oplus A_4^{\eps_4}
 \oplus(A_4^*)^{\eps_{4^*}}
 \oplus A_5^{\eps_5}\oplus A_6^{\eps_6},
\end{equation*}
where $\eps_2+\eps_4+\eps_{4^*}+\eps_5+\eps_6\geq1.$
\end{theorem}

\begin{proof}
Recall from Proposition~\ref{prop:fundamental-polynomial-growth} that a fundamental algebra of
polynomial codimension growth is either nilpotent, isomorphic to $F$ or
of the form $F+J(A)$. Since $s_A=2$, the codimension
sequence of $A$ has quadratic growth, see for instance Theorem \ref{thm:fundamental-growth-degree}. Therefore, $A$ is neither nilpotent nor isomorphic to
$F$. Consequently, $A=F+J$, where $J=J(A).$ Moreover, the definition of $s_A$ gives $J^3=\{0\}$ and $    J^2\neq\{0\}.$ 

Let $B(A)$ be the detector algebra defined in \eqref{eq:detector-sum}. 
\begin{comment}
{\color{purple} We first prove that
$\Id(A)\subseteq\Id(B(A)).$ By Lemma~\ref{A1A2A4}, the nonvanishing of $J_{10}$, $J_{01}$, $[J_{11},J_{11}]$, $J_{10}J_{00}$ and $J_{00}J_{01}$ implies, respectively, that $A_1$, $A_1^*$, $A_2$, $A_4$, $A_4^*
    \in\var(A).$ Similarly, Lemma~\ref{lem:A5A6-detectors} shows that if $J_{01}J_{10}\neq\{0\}$ then $A_5\in\var(A)$, and if
$J_{10}J_{01}\neq\{0\}$ then $A_6\in\var(A).$ By the definition of the parameters $\varepsilon_i$, these are precisely the nonzero summands occurring in $B(A)$. Consequently, every nonzero summand occurring in $B(A)$ belongs to
$\var(A)$. Since varieties are closed under finite direct sums, it
follows that $B(A)\in\var(A),$ which gives $\textnormal{Id}(A)\subseteq \textnormal{Id}(B(A))$.} 

{\color{blue} Essa afirmação, $\textnormal{Id}(A)\subseteq \textnormal{Id}(B(A))$, já é verdadeira pelos comentários feitos após a definição de B(A). Sugiro que essa parte seja retirada.}
\end{comment}
We next show that $B(A)$ contains at least one quadratic model. Suppose,
to the contrary, that
\[
 [J_{11},J_{11}]
 =
 J_{10}J_{00}
 =
 J_{00}J_{01}
 =
 J_{01}J_{10}
 =
 J_{10}J_{01}
 =\{0\}.
\]
Lemma~\ref{lem:no-quadratic-mechanism} would then imply $c_n(A)\leq  \alpha n$, for some constant $\alpha$ and for all $n\geq 1$. On the other hand, since $A$ is fundamental and $s_A=2$,
Corollary~\ref{cor:quadratic-growth-criterion} gives a constant $\gamma>0$ such
that $c_n(A)\geq \gamma n^2$ for all sufficiently large $n$, a contradiction. Therefore
$\eps_2+\eps_4+\eps_{4^*}+\eps_5+\eps_6\geq1.$

We can now apply Lemma~\ref{lem:transfer} to obtain the inclusion $\Id(B(A))\subseteq\Id(A).$
Combining this result with Remark \ref{rem:B(A)}, we obtain
$\Id(A)=\Id(B(A)),$ and hence $A\simeqT B(A)$. This proves
the result.
\end{proof}

Two observations concerning the detector algebra $B(A)$ are in order. 

First, the direct-sum presentation of $B(A)$ is not necessarily
redundancy-free. For instance, suppose that
$\varepsilon_1=\varepsilon_4=1$. Since
$A_1\in\operatorname{var}(A_4),$
we have $\operatorname{Id}(A_4)\subseteq\operatorname{Id}(A_1)$, and
therefore
\[
\begin{aligned}
    \operatorname{Id}(A_1\oplus A_4)
      =\operatorname{Id}(A_1)\cap\operatorname{Id}(A_4)=\operatorname{Id}(A_4).
\end{aligned}
\]
Consequently, $A_1\oplus A_4\sim_T A_4,$
so the summand $A_1$ does not affect the $T$-ideal of $B(A)$ in this
case. Similarly, since $A_1^*\in\operatorname{var}(A_4^*),$
the conditions $\varepsilon_{1^*}=\varepsilon_{4^*}=1$ imply $A_1^*\oplus A_4^*\sim_T A_4^*.$ Thus, the detector algebra records the nonzero products of the radical components, but its presentation as a direct sum need not
be minimal with respect to PI-equivalence.

Second, although the preceding theorem yields $A\sim_T B(A),$ the direct sum $B(A)$ is not required to be fundamental. Nevertheless,
all its nonzero summands are fundamental and have semisimple parts
isomorphic to $F$. The next proposition shows that these summands can be
successively fused to produce a fundamental algebra $D$ satisfying $D\sim_T B(A).$
Moreover, since $B(A)$ contains at least one quadratic model, the
resulting algebra can be chosen with radical parameter $s_D=2$.

\begin{proposition}
\label{prop:converse}
Let $B$ be the direct sum of the members of a subset of
$\{A_1,A_1^*,A_2,A_4,A_4^*,A_5,A_6\},$ and assume that at least one of
$ A_2$, $A_4$, $A_4^*$, $A_5$ and $A_6$
occurs as a summand. Then there exists a fundamental algebra $D$ such
that $s_D=2$ and $D\simeqT B.$ Consequently, every PI-equivalence class appearing in
Theorem~\ref{thm:main} is realized by a fundamental algebra with
$s_D=2$.
\end{proposition}

\begin{proof} By Proposition \ref{prop:quadratic-algebras-fundamental},
each algebra in $\{A_1,A_1^*,A_2,A_4,A_4^*,A_5,A_6\}$ is fundamental and has semisimple part isomorphic to $F$.  Apply successively the fusion construction for fundamental algebras having isomorphic semisimple parts, see \cite[Proposition~4.3]{GQV2026} and also \cite[Lemma 2.17]{AK2024}. This produces a fundamental algebra
$D$ satisfying $D\simeqT B.$ Moreover, the fusion construction preserves the maximum of the radical
parameters of its factors. Hence
\[
 s_D
 =
 \max\{s_M:M\text{ is a summand of }B\}.
\]
Since $s_{A_1}=s_{A_1^*}=1$ and every quadratic model has radical parameter equal to $2$, the assumption that $B$ contains a quadratic summand gives $s_D=2$.
\end{proof}
\begin{comment}
{\color{purple} Observe que a proposição anterior, assim como \cite[Proposition~4.3]{GQV2026}, apenas demonstra a existência dessa álgebra fundamental PI-equivalente à soma direta de fundamentais, sem apresentar um representante para a mesma. Em \cite{GQV2026}, estendemos um pouco essa idéia e as do Aljadeff e do Karasik e mostramos como construir essas álgebras fundamentais PI-equivalentes a somas diretas de outras fundamentais. Essa nova fundamental será uma álgebra de matrizes onde amarramos alguns termos. A seguir eu escrevo essas informações como uma breve observação que sugiro acrescentarmos.}
\end{comment}

\begin{remark} Notice that the preceding proposition establishes the existence of the fundamental algebra, but does not provide an explicit representative for it. However, using the ideas contained in \cite{GQV2026-2}, one can construct such a fundamental algebra as follows. Let $A\subseteq M_k(F)$ and $B\subseteq M_l(F)$ be two fundamental algebras with $\overline{A}\cong\overline{B}\cong F$. By \cite[Proposition~4.3]{GQV2026}, there exists a fundamental algebra $D\sim_{T}A\oplus B$. To construct $D$, write $A=Fe_A+J(A)$ and $B=Fe_B+J(B)$,
where $e_A$ and $e_B$ are the identities of their respective
semisimple parts, and set
\[
D=
\left\{
\begin{pmatrix}
\alpha e_A+u & 0\\
0 & \alpha e_B+v
\end{pmatrix}
\;\middle|\;
\alpha\in F,\ u\in J(A),\ v\in J(B)
\right\}
\subseteq M_{k+\ell}(F).
\]
Thus, the two semisimple components are replaced by a single diagonal copy of $F$,
while retaining both radicals unchanged.
Since $D\subseteq A\oplus B$ and the block projections map
$D$ onto $A$ and $B$, we obtain $D\sim_T A\oplus B$.
Moreover, $J(D)=\operatorname{diag}(J(A),J(B))$, so
$s_D=\max\{s_A,s_B\}$.
By \cite[Proposition~4.3]{GQV2026}, $D$ is fundamental.
\end{remark}

\black 

We can now state the complete classification of varieties with quadratic codimension growth.

\begin{comment}
{\color{purple}Acredito que a próxima demonstração pode ser simplificada, observe que, por \cite[Proposition 4.8]{GQV2026}, toda álgebra $A$ com crescimento polinomial é PI-equivalente a uma soma direta $B \oplus N$, onde $B$ é fundamental. Se $A$ tem crescimento quadrático, então $B$ é uma fundamental com crescimento quadrático. Logo, o resultado segue pelo Teorema \ref{thm:main} juntamente com o Corolário \ref{cor:quadratic-growth-criterion}. Além disso, pelos comentários feitos após o Teorema \ref{thm:bounded-colength}, já está claro que o corpo poderia ser assumido como algebricamente fechado. Vou escrever a seguir a sugestão para o próximo corolário.}{\color{blue}
\end{comment} 

\begin{corollary}
\label{cor:classification-algebras}
Let \(A\) be an algebra over a field \(F\) of characteristic zero.
Then \(A\) has quadratic codimension growth if and only if there
exist a nilpotent \(F\)-algebra \(N\) and parameters $\varepsilon_1,\varepsilon_{1^*},\varepsilon_2,
\varepsilon_4,\varepsilon_{4^*},\varepsilon_5,\varepsilon_6
\in\{0,1\},$ with $\varepsilon_2+\varepsilon_4+\varepsilon_{4^*}
+\varepsilon_5+\varepsilon_6\geq1,$ such that
\[
A\sim_T
N\oplus A_1^{\varepsilon_1}
\oplus(A_1^*)^{\varepsilon_{1^*}}
\oplus A_2^{\varepsilon_2}
\oplus A_4^{\varepsilon_4}
\oplus(A_4^*)^{\varepsilon_{4^*}}
\oplus A_5^{\varepsilon_5}
\oplus A_6^{\varepsilon_6}.
\]
\end{corollary}

\begin{proof}
Let \(K=\overline F\) and set \(A_K=K\otimes_F A\). Since
codimensions are preserved under scalar extension, \(A_K\) has
quadratic codimension growth. By Corollary~\ref{cor:finite-dimensional-reduction}, there exists a
finite-dimensional \(K\)-algebra \(E\) such that $A_K\sim_T E.$ Applying \cite[Proposition 4.8]{GQV2026} to \(E\),
and taking the nilpotent summand to be zero when \(E\) is already
fundamental, we obtain $E\sim_T H\oplus N_K,$ where \(H\) is fundamental and \(N_K\) is nilpotent.

Choose \(d\geq1\) such that \(N_K^d=0\). For every \(n\geq d\),
every multilinear polynomial of degree \(n\) is an identity of
\(N_K\). Hence $P_n\cap\operatorname{Id}(H\oplus N_K)
=
P_n\cap\operatorname{Id}(H),$ and therefore $c_n(E)=c_n(H)$, for all $n\geq d.$ It follows that \(H\) is nonnilpotent and has quadratic codimension
growth. By Corollary~\ref{cor:quadratic-growth-criterion}, \(s_H=2\).
Theorem~\ref{thm:main} now gives
\[
H\sim_T
A_{1,K}^{\varepsilon_1}
\oplus(A_{1,K}^*)^{\varepsilon_{1^*}}
\oplus A_{2,K}^{\varepsilon_2}
\oplus A_{4,K}^{\varepsilon_4}
\oplus(A_{4,K}^*)^{\varepsilon_{4^*}}
\oplus A_{5,K}^{\varepsilon_5}
\oplus A_{6,K}^{\varepsilon_6},
\]
where $A_{i,K}:=K\otimes_F A_i$ and at least one quadratic parameter is nonzero.

Regard \(N_K\) as an \(F\)-algebra by restriction of scalars and
denote it by \(N\). It remains nilpotent. Since \(N\) and \(N_K\)
have the same underlying ring,
$\operatorname{Id}_F(N)
=
\operatorname{Id}_K(N_K)\cap F\langle X\rangle.$ Moreover, for every model algebra \(A_i\),
$\operatorname{Id}_F(A_i)
=
\operatorname{Id}_K(A_{i,K})\cap F\langle X\rangle.$ Contracting the equality of \(T\)-ideals over \(K\) to
\(F\langle X\rangle\) gives the required PI-equivalence over \(F\).

Conversely, suppose that \(A\) is PI-equivalent to the displayed direct sum. Since PI-equivalent algebras have the same codimension sequence, it is enough to consider that direct sum. At least one quadratic model occurs, while every remaining model algebra has at most quadratic codimension growth and the nilpotent summand has eventually zero codimensions. Therefore, \(A\) has quadratic
codimension growth.
\end{proof}

\black

Recall that varieties of at most linear codimension growth were shown to be generated by finite direct sums of algebras generating minimal varieties of at most linear growth, possibly together with a nilpotent algebra (see \cite[Theorem 22]{GLM2005}). Combining this result with our classification of varieties of quadratic growth, we obtain a unified description of all varieties of at most quadratic growth: they are generated by finite direct sums of algebras generating minimal varieties of at most quadratic growth, again possibly together with a nilpotent algebra. This yields the following corollary.

\begin{corollary}
Let $\mathcal V$ be a variety of associative algebras over a field $F$ of characteristic zero. Then there exists a constant \(\alpha>0\) such that
$c_n(\mathcal V)\leq \alpha n^2,$ for every $ n\geq 1,$ if and only if one of the following alternatives holds:
\begin{enumerate}

\item $\mathcal V=\operatorname{var}(N)$ for some nilpotent algebra
$N$;
\item $\mathcal V=\operatorname{var}(F\oplus N)$ for some nilpotent
algebra $N$;
\item there exist a nonempty subset
$S\subseteq\{A_1,A_1^*,A_2,A_4,A_4^*,A_5,A_6\}$ and a nilpotent algebra $N$ such that
\[
\mathcal V=\operatorname{var}\left(
N\oplus\bigoplus_{M\in S}M\right).
\]
\end{enumerate}
Moreover, in item $(3)$, the variety \(\mathcal V\) has
quadratic growth if and only if $S\cap\{A_2,A_4,A_4^*,A_5,A_6\}\neq\varnothing.$
\end{corollary}

\end{document}